\pdfoutput=1
\documentclass[twoside, 11pt]{article}
\usepackage[colorlinks]{hyperref}
\usepackage[utf8]{inputenc}
\usepackage[T1]{fontenc}
\usepackage[charter]{mathdesign}
\usepackage{amsthm, amsfonts, enumerate}
\usepackage[a4paper,margin=2.9cm, headsep=15pt, headheight=1.5cm]{geometry}
\usepackage{spectralsequences}

\newcommand\smvee{\raise0.3ex\hbox{$\scriptscriptstyle\vee$}}
\usepackage{fancyhdr}
\usepackage{tikz-cd, tikz}
\usetikzlibrary{matrix, calc}

\usepackage{graphicx, mathtools, float}

\newtheorem{theorem}{Theorem}[section]
\newtheorem{nmtheorem}{Theorem}
\newtheorem{mmtheorem}{Theorem}

\newtheorem{remark}[nmtheorem]{Remark}

\newtheorem{cor}[mmtheorem]{Corollary}

\numberwithin{equation}{section}
\newcounter{para}

\usepackage[nameinlink]{cleveref}
\hypersetup{colorlinks,breaklinks,
	citecolor=[rgb]{0.1,0.7,0.5},
	linkcolor=[rgb]{0.8,0.2,0.1}}

\def\C{\mathbb{C}}
\def\Cc{\mathbb{C}^{\times}}

\def\Qb{\mathbb{Q}}

\def\P{\mathbb{P}}
\def\Pc{\mathbb{P}_{\mathbb{C}}}

\def\F{\mathbb{F}}

\def\S{\mathit{Sym}}

\def\pconf{\mathit{PConf}}
\def\uconf{\mathit{UConf}}

\newcommand\Sym{\mathit{Sym}}

\date{}

\newcount \questionno
\def\question{\medbreak
	\global \advance \questionno 1
	\textbf{Problem \the\questionno}.\enspace \ignorespaces}
\newcommand\shorttitle{On the cohomology of certain subspaces of $\Sym^n(\P^1)$ and Occam's razor for Hodge structures}
\newcommand\authors{Oishee Banerjee}
\usepackage{titlesec}

\title{On the cohomology of certain subspaces of $\Sym^n(\P^1)$ and Occam's razor for Hodge structures}
\author{Oishee Banerjee}
\begin{document}
	\maketitle
	\begin{abstract}
	 In \cite{Vakil13} Vakil and Wood made several conjectures on the topology of symmetric powers of geometrically irreducible varieties based on their computations on motivic zeta functions. Two of those conjectures are about subspaces of $\Sym^n(\P^1)$. In this note, we disprove one of them thereby obtaining a counterexample to the principle of Occcam's razor for Hodge structures; and we prove that the other conjecture, with a minor correction, holds true.
	\end{abstract}
\section{Introduction}
For a smooth and proper variety $X$ over $\C$, the Hodge-Deligne polynomial determines the Hodge numbers; but that is no longer the case when $X$ is not smooth and proper. To elaborate, for any variety $X$ over $\C$, the compactly supported cohomology groups $H^i_c(X,\Qb)$ carry Deligne's mixed hodge structures. One defines the \emph{Hodge-Deligne polynomial} as $$HD(x,y):= \sum\limits_{p,q}e_{p,q}x^py^q.$$
Here $e_{p,q}$ are \emph{virtual Hodge-Deligne numbers}, defined in terms of pure Hodge structures that the associated gradeds for the weight filtration on $H^*_c(X,\Qb)$ are equipped with:
$$e_{p,q} = \sum\limits_{i} (-1)^ih^{p,q}\Big(\mathit{gr}_{W}^{p+q} H^i_c(X,\Qb)\Big).$$ When $X$ is smooth and proper, one has $e_{p,q} = (-1)^i h_{p,q}(H^i(X,\Qb)).$ There are many examples where the simplest possibility holds i.e. there is a simplest Hodge structure on $H^i_{c}(X,\Qb)$ for all $i$ in agreement with the virtual Hodge structure. In \cite{Vakil13}, Vakil and Wood dub this well-known principle as "Occam’s razor for Hodge structures".
This principle led them to conjecture about the stable rational cohomology of certain subspaces of $\Sym^m(\P^1)$, the $m$-fold symmetric product of $\Pc^1$. These are Conjectures G' and H' in \cite{Vakil13} \footnote{Note that the conjectures are in the arxiv version of the paper, and not the published version \cite{Vakil15}}. The goal of this note is to disprove one of them, and prove a slight alteration of the other.

Before stating their conjectures and the main theorem of this note, we fix some notations. All varieties are over $\mathbb{C}$. For a complex quasiprojective variety $X$, let $\Sym^m(X)$ denote the $n$-fold symmetric product. For a partition $\lambda$ of $m$ and a complex variety $X$, let $w_{\lambda}(X)$ denote the locally closed subset of $\Sym^m (X)$ with multiplicities precisely $\lambda$. Further more, let $\uconf_n X$ denote the unordered configuration space of $n$ points on $X$ i.e. $\uconf_n X= w_{1^n}(X)$, and let $\pconf_n X$ denote the ordered configuration space of $n$ points on $X$. 

	Conjecture H' of \cite{Vakil13} states that the values of $i>0$ for which $$\lim\limits_{n\to \infty} \dim H^i(w_{1^n22}(\P^1);\Qb) \neq 0$$ in \eqref{eq1.1} is periodic in $i$, and the nonzero limits equal $1$. Conjecture G' of \cite{Vakil13} states that $$	\lim\limits_{n\to \infty} \dim H^i(w_{1^n23}(\P^1);\Qb) =\begin{dcases}1 & i=0,1,\\0 & \text{ otherwise.}\end{dcases}$$

Our goal is to prove the following theorem.
\begin{mmtheorem}\label{thmA}
	Let $n\geq 2$. Then \begin{gather*}
	H^i(w_{1^n22}(\P^1);\Qb) = \begin{dcases}
	\Qb & \text{for } i=0, n \\
	\Qb^{2} & \text{for } 1\leq i\leq n-1,\\
	0 & \text{otherwise.} 
	\end{dcases}
	\end{gather*}
	In particular, for all $i\geq 1$, \begin{gather}
	\lim\limits_{n\to \infty} \dim H^i(w_{1^n22}(\P^1);\Qb) = 2.\label{eq1.1}
	\end{gather} Furthermore, $H^i(w_{1^n22}(\P^1);\Qb)$  is pure of weight $-2i$ and Hodge type $(-i,-i)$ for all $0\leq i\leq n$. \hfill$\square$
\end{mmtheorem}

\noindent The following corollary to Theorem \ref{thmA} disproves Conjecture G'.

\begin{cor}\label{cor}
	Let $n\geq 2$. Then for all $0\leq i\leq n$, \begin{gather}
	\lim\limits_{n\to \infty} \dim H^i(w_{1^n23}(\P^1);\Qb) \neq 0.\label{eq1.2} 
	\end{gather} \hfill$\square$
\end{cor}

	\begin{remark}\textup{
		A question along the lines of the conjectures based on the Occam's razor of Hodge structures would be, can one determine the rational cohomology of a variety over $\C$ by counting the number of $\mathbb{F}_q$ points of that variety? The answer, in general, is in negative. In fact, the conjectures were made on the basis of such point-counts. The Grothendieck-Lefschtez trace formula (see \cite{Grothendieck1965}) allows one to count the number of $\mathbb{F}_q$ points of a variety $X$ from its topology, when$X$ is a reasonably nice variety. However, there is no sufficient criterion to cross the bridge from $\F_q$ points of a variety to the rational Betti numbers of the topological space formed by its $\C$-points. This note provides two such examples.}
	\end{remark}
		\subsection*{Acknowledgements}
	I am grateful to my advisor, Benson Farb for his helpful comments and patient guidance. I also deeply thank Melanie Matchett-Wood for her valuable feedback on an earlier draft of the manuscript.

\section{Cohomological stability of some locally closed strata of $\S^m(\P^1_{\C})$}
In this section we prove Theorem \ref{thmA} and Corollary \ref{cor}. One of the important steps is to compute $H^*(\uconf_n\C^{\times};\Qb)$. The latter quantity is well-known (see e.g. \cite{Schiessl2018}, \cite{Drummond-Cole2017} and the references therein); in this paper, we will heavily use the notion of spaces admitting a semi-filtration developed in \cite{Banerjee2019} to give a short alternative method to compute $H^*(\uconf_n\C^{\times};\Qb)$. Our proof of Theorem \ref{thmA} can be outlined via the following steps. \begin{enumerate}
	\item Describe the space $w_{1^n22}(\P^1)$ is a fibre-bundles over $\uconf_2(\Pc^1)$ with fibres isomorphic to $\uconf_n(\Cc)$.
	\item Invoke \cite[Corollary 2]{Banerjee2019} to compute $H^*_c(\uconf_n\C^{\times};\mathbb{Q})$, the compactly supported rational cohomology of $\uconf_n \C^{\times}$.  Then use the Poincaré duality and the Serre spectral sequence for a fibration, in that order, to compute $H^*(w_{1^n22}(\P^1);\Qb)$.
\end{enumerate} 

 \begin{proof}[Proof of Theorem \ref{thmA}]
 	
 	Let us address Step 2 first. We compute $H^*_c(\uconf_n\C^{\times};\mathbb{Q})$ using Corollary 2 from \cite{Banerjee2019}, which we state here for the sake of completeness:
 	
 		Let $X$ be a connected locally compact Hausdorff topological space. Then there exists a spectral sequence: \begin{gather}\label{eq2.1}
 	E_1^{p,q} =\bigoplus_{l+m=q}\,\,\bigoplus_{i+j=p} \big(\Sym^i H_c^{\text{odd}}(X;\mathbb{Q}) \otimes \Lambda^j H_c^{\text{even}}(X;\mathbb{Q})\big)^{(l)} \otimes H^m_c(\Sym^{n-2p}X;\mathbb{Q}) \implies H^{p+q}_c(\uconf_n(X);\mathbb{Q})
 	\end{gather} where $ H_c^{\text{odd}}(X;\mathbb{Q}) := \bigoplus_{k} H_c^{2k+1}(X;\mathbb{Q})$ and $ H_c^{\text{even}}(X;\mathbb{Q}) := \bigoplus_{k} H_c^{2k}(X;\mathbb{Q})$, and $$\big(\Sym^i H_c^{\text{odd}}(X;\mathbb{Q}) \otimes \Lambda^j H_c^{\text{even}}(X;\mathbb{Q})\big)^{(l)} $$ denotes the $l^{(\mathrm{th})}$-graded summand of the cohomology  $\Sym^i H_c^{\text{odd}}(X;\mathbb{Q}) \otimes \Lambda^j H_c^{\text{even}}(X;\mathbb{Q})$. Put $X=\C^{\times}$ in \eqref{eq2.1}.  Then for  $p\geq 1$, the spectral sequence \eqref{eq2.1} gives us: \begin{gather}\label{eq2.2}
E_1^{p,q} = \begin{cases}
\Sym^{p-1} H^1_c(\Cc;\Qb) \otimes H^2_c(\Cc;\Qb) \otimes \Sym^{n-2p}H^2_c(\Cc;\Qb) & q= 2n-3p+1,\\
\Sym^p H^1_c(\Cc;\Qb) \otimes \Sym^{n-2p}H^2_c(\Cc;\Qb) \bigoplus \\\Sym^{p-1}H^1_c(\Cc;\Qb) \otimes H^2_c(\Cc;\Qb) \otimes \Sym^{n-2p-1}H^2_c(\Cc;\Qb)\otimes H^1_c(\Cc;\Qb)& q=2n-3p,\\
\Sym^{n-2p-1}H^2_c(\Cc;\Qb)\otimes H^1_c(\Cc;\Qb) \otimes \Sym^p  H^1_c(\Cc,\Qb) & q=2n-3p-1,\\0 & \text{ otherwise,}
\end{cases}
 	\end{gather}
and for $p=0$ one has  \begin{gather*} 
E_1^{0,q}=\begin{cases}
\Sym^{n} H^2_c(\Cc;\Qb) & q=2n,\\
\Sym^{n-1} H^2_c(\Cc;\Qb)\otimes H^1_c(\Cc;\Qb) & q=2n-1,
\end{cases}
\end{gather*}	with the differentials going horizontally $E_1^{p,q} \to E_1^{p+1,q}$ (see Figure \ref{fig}). Furthermore, one can read off the weights from the explicit description of the terms $E_1^{p,q}$ of the spectral sequence in \eqref{eq2.2} by noting that $H^1_c(\Cc;\Qb)$ is pure of weight $-2$ and Hodge type $(-1,-1)$.  Letting $\Qb(1)$ denote the Tate Hodge structure of weight $-2$ and Hodge type $(-1,-1)$, we obtain:   \begin{gather}H^i(\uconf_n\Cc;\Qb)\cong \begin{cases}
\Qb(i)& i= 0,n,\\
\Qb(i)^2 & 0<i<n.
\end{cases}
\end{gather}
 	\begin{figure}[H]
 		\begin{tikzpicture}          
 		\node (A1) at (1.5, 3.3) {$ $};     \node at (1.3,3) {$2n$}; \node at (1.2,2) {$2n-1$};  \node at (1.2,1) {$2n-2$}; \node at (1.1,0) {$2n-3$}; 
 		\node at (1.1,-1) {$2n-4$};  \node at (1.1,-2) {$2n-5$}; \node at (1.1,-3) {$2n-6$}; \node at (-0.7,-1.8) {$q\uparrow$}; \node at (4.4,-10.7) {$p\rightarrow$}; 
 		\node (A2) at (1.5,-10.1) {$ $};  \node (A3) at (1.4,-10) {$ $};  \node (A4) at (7,-10) {$ $};  \node at (7.5,-10) {$\ldots \ldots$};   
 		\draw [-] (A1)-- (A2);         \draw [-] (A3)-- (A4);                  
 		\node at (2.5,3) {$\mathbb{Q}$};     \node at (2.5,2) {$\mathbb{Q}$};        \node at (2.5,1) {$0$};         \node at (2.5,0) {$0$};   
 		\node at (3.5,3) {$0$};  \node at (3.5,2) {$0$};  \node at (3.5,1) {$\mathbb{Q}$};     \node at (3.5,0) {$\mathbb{Q}^2$};      \node at (3.5,-1) {$\mathbb{Q}$};  \node at (3.5,-4) {$0$}; \node at (3.5,-2) {$0$}; \node at (3.5,-3) {$0$};      
 		\node at (4.5,1) {$0$};  \node at (4.5,0) {$0$};  \node at (4.5,-1) {$0$}; \node at (4.5,-2) {$\mathbb{Q}$};  \node at (4.5,-3) {$\mathbb{Q}^2$};   \node at (4.5,-4) {$\mathbb{Q}$};    \node at (4.5,-5) {$0$};   
 		\node at (5.5,-2) {$0$};   \node at (5.5,-3) {$0$};  \node at (5.5,-4) {$0$};   \node at  (5.5,-5) {$\mathbb{Q}$};   
 		\node at  (5.5,-6) {$\mathbb{Q}$};       \node at  (5.5,-7) {$\mathbb{Q}$};     
 		\node at (2.5,-10.2) {$0$};      \node at (3.5,-10.2) {$1$};      \node at (4.5,-10.2) {$2$};      \node at (5.5,-10.2) {$3$};                                               
 		\end{tikzpicture}
\caption{$E_1$ page spectral sequence converging to $H^*_c(\uconf_n(\Cc);\Qb)$.}\label{fig}
 	\end{figure} 

Now we work out Step 1 from the proof outline. For any positive integer $n$ define the map  \begin{gather} \pi_{22}: w_{1^n22}(\P^1)\to \uconf_2(\P^1)\nonumber \\ \Big(\{x_1,\ldots, x_n\}, a,a,b,b\Big)\mapsto \{a,b\}, \label{eq2.3}\end{gather} and note that for all $\{a,b\}\in\uconf_2(\P^1)$, $$\pi_{22}^{-1}\{a,b\} =\uconf_n (\P^1-\{a,b\}) \cong \uconf_n(\Cc).$$ An equivalent description of $w_{1^n22}(\P^1)$ is that it's a quotient of the fibre bundle \begin{gather*}
F_2: \pconf_{n+2}(\P^1) \to \pconf_2(\P^1)\\
(x_1,\ldots, x_{n+2}) \mapsto (x_{n+1},x_{n+2})
\end{gather*}  by the action of $\mathfrak{S}_{n}$ on the fibres of  $F_2$, where \begin{gather*}
F_2^{-1}(x,y) =\{(x_1,\ldots,x_n,x,y):x_i\in \P^1-\{x,y\}, x_i\neq x_j \text{ for } i\neq j \}\\\cong\pconf_n(\Cc),\end{gather*} and the action of $\mathfrak{S}_2$ on the base $\pconf_2(\P^1)$. Let $PB_n:= \pi_1(\pconf_n(\P^1))$ be the \emph{pure Hurwitz braid group} on $n$ strands. Then $PB_2$ acts trivially on the homology of the fibres of $F_2$ because it acts by conjugation on $\pi_1(\uconf_n(\Cc))$.
Therefore, the Hurwitz braid group on two strands $\pi_1(\uconf_2(\P^1))$ also acts by conjugation on $\pi_1(\uconf_n(\Cc))$, so the monodromy is trivial on the homology of the fibres of $\pi_{22}$ in \eqref{eq2.3}. 

On the other hand $H^*_c(\uconf_2(\P^1);\Qb) \cong \Qb$. One way to see this is by using the long-exact sequence of cohomology. Equivalently, plugging $X = \P^1$ and $n=2$ in \eqref{eq2.1} (see \cite[Corollary 2]{Banerjee2019}), we get the following spectral sequence: \begin{gather}\label{eq2.4}
E_1^{p,q} = \begin{cases}
\Qb & (p,q) = (0,0), (0,2), (0,4), (1,0), (1,2)\\
0 &  \text{ otherwise.}
\end{cases}
\end{gather}

\begin{sseqdata}[ name = basic, cohomological Serre grading, classes = {draw=none}]
	\class["\Qb"](0,0)
	\class["\Qb"](0,2)
	\class["\Qb"](0,4)
	\class["\Qb"](1,0)
	\class["\Qb"](1,2)
	\d1(0,0)
	\d1 (0,2)
\end{sseqdata}
\printpage[ name = basic ]

The differentials in \eqref{eq2.4} are induced by the diagonal map \begin{gather}
\Delta: \P^1 \to \Sym^2(\P^1)\nonumber\\ x\mapsto \{x,x\}
\end{gather} To understand the differentials we need to compute the induced map on cohomology $$\Delta^*: H^*(\Sym^2(\P^1);\Qb) \to H^*(\P^1;\Qb).$$
To this end, we think of $\Sym^2(\P^1)$ as the space of degree $2$ divisors in $\P^1$; so we have an isomorphism $$F: \mathbb{P}(\Gamma(\P^1,\mathcal{O}_{\P^1}(2))^{\smvee}) \xrightarrow{\cong}  \Sym^2(\P^1) $$ given by a global section in $\mathcal{O}_{\P^1}(2)$ mapping to its divisor (for a vector space $V$, we denote its dual by $V^{\smvee}$).  Note that $\mathbb{P}(\Gamma(\P^1,\mathcal{O}_{\P^1}(2))^{\smvee})\cong \P^2$. In terms of coordinates one can write down $F^{-1}$ as:\begin{gather}
F^{-1}:\Sym^2 (\P^1) \xrightarrow{\cong} \P^2\nonumber\\ \{[a_1:b_1], [a_2,b_2]\} \mapsto [a_1a_2: -(a_1b_2+a_2b_1): b_1b_2]
\end{gather}
Now $$F^{-1}_{\circ} \Delta: \P^1 \to \P^2$$ embeds $\P^1$ as a smooth conic in $\P^2$; it is the discriminant locus cut out by $y^2-xz=0$. And observe that $$(F^{-1}_{\circ} \Delta)^*: H^i(\P^2;\Qb) \to H^i(\P^1;\Qb)$$ is an isomorphism for $i=0,2$. Indeed, the fundamental class of $\P^2$ restricts to that of $\P^1$ for $i=0$, and for $i=2$ the hyperplane class in $H^2(\P^2;\mathbb{Z}) $ restricts to twice the class of a point in $\P^1$ by Bézout's theorem, thereby inducing an isomorphism on cohomology with $\Qb$ coefficients in degree $2$. Combining this information with the spectral sequence in \eqref{eq2.4} we get that \begin{gather}\label{eq2.7}
H^i(\uconf_2(\P^1);\Qb) \cong \begin{cases}
\Qb & i=4\\
0 & \text{ otherwise.}
\end{cases}
\end{gather}

The Serre spectral sequence for the fibration \eqref{eq2.3}, combined with \eqref{eq2.2}, \eqref{eq2.7}, and the fact that $\pi_1(\uconf_2(\P^1))$ acts trivially on $H_*(\uconf_n(\Cc);\Qb)$, gives us \begin{gather}E_2^{p,q} = \begin{dcases}
H^0(\uconf_2(\P^1);\Qb)\otimes H^q(\uconf_n(\Cc);\Qb) &p=0\\
0 & p\geq 0
\end{dcases}\implies H^*(w_{1^n22}(\P^1);\Qb) .\end{gather}
Therefore for all $i\geq 0$ we have $$H^i(w_{1^n22}(\P^1);\Qb) \cong H^i(\uconf_n(\Cc);\Qb)$$ proving Theorem \ref{thmA}.\end{proof}

\begin{proof}[Proof of Corollary \ref{cor}]
	The space $w_{1^n23}(\P^1)$ is a fibre bundle: 	\begin{gather} \pi_{23}: w_{1^n23}(\P^1)\to \pconf_2(\P^1)\nonumber\\ \{x_1,\ldots, x_n\}, a,a,b,b,b\mapsto (a,b) \label{eq2.10} \end{gather} with fibres $$\pi_{23}^{-1}(a,b) =\uconf_n (\P^1-\{a,b\})\cong \uconf_n(\Cc).$$ An equivalent description of $w_{1^n23}(\P^1)$ is that it's a quotient of the fibre bundle \begin{gather*}
	F_2: \pconf_{n+2}(\P^1) \to \pconf_2(\P^1)\\
	(x_1,\ldots, x_{n+2}) \mapsto (x_{n+1},x_{n+2})
	\end{gather*}  by the action of $\mathfrak{S}_{n}$ on the fibres of  $F_2$, where \begin{gather*}
	F_2^{-1}(x,y) =\{(x_1,\ldots,x_n,x,y):x_i\in \P^1-\{x,y\}, x_i\neq x_j \text{ for } i\neq j \}\\\cong\pconf_n(\Cc).\end{gather*} Therefore, $w_{1^n23}(\P^1) /\mathfrak{S}_2 \cong w_{1^n22}(\P^1)$, and in turn, \begin{gather}H^*(w_{1^n22}(\P^1);\Qb) \cong \Big(H^*(w_{1^n23}(\P^1);\Qb)\Big)^{\mathfrak{S}_2} \end{gather} By Theorem \ref{thmA}, for each $0\leq i \leq n$ the dimension of the $\Qb$-vector space $ H^i(w_{1^n22}(\P^1);\Qb)$, and in turn $\Big(H^i(w_{1^n22}(\P^1);\Qb)\Big)^{\mathfrak{S}_2} $ is nonzero, therefore $\dim H^i(w_{1^n23}(\P^1);\Qb) \neq 0$ for $0\leq i\leq n$, thus proving the corollary.
\end{proof} 
\begin{remark}
	Note that using the Serre spectral sequence for the fibre bundle in \eqref{eq2.10} one obtains \begin{gather*}
E_2^{p,q} \cong \begin{cases}
\Qb & \{(p,q): p=0,2, \,\, q=0,n\}\\
\Qb^2 & \{(p,q): p=0,2, \,\, 0<q<n\}\\ 0 & \text{ otherwise.}
\end{cases}\end{gather*} as shown in the following figure:

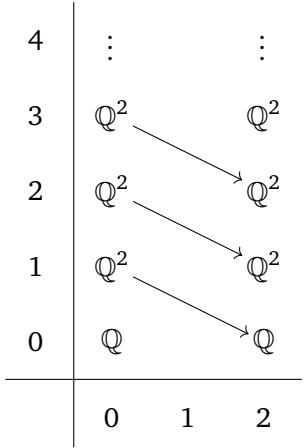
\begin{sseqpage}[cohomological Serre grading, classes = {draw=none}]
	\class["\Qb"](0,0)
	\class["\Qb^2"](0,1)
	\class["\Qb^2"](0,2)
		\class["\Qb^2"](0,3)
	\class["\vdots"](0,4)
	\class["\Qb"](2,0)
	\class["\Qb^2"](2,1)
	\class["\Qb^2"](2,2)
		\class["\Qb^2"](2,3)
			\class["\vdots"](2,4)
	\d2 (0,1)
	\d2 (0,2)
\d2 (0,3)
\end{sseqpage}

If the differentials $d_2^{p,q}$ were isomorphisms (respectively, surjection) of $\Qb$-vector spaces for $p=0$ and $q\geq 2$ (respectively, $q=1$), we would have had $$H^i(w_{1^n23}(\P^1);\Qb)=\begin{cases}
\Qb & i=0,1\\ 0 & \text{ otherwise,}
\end{cases}$$ and this is exactly the statement of Conjecture H' of Vakil-Wood.
However, the proof of Corollary \ref{cor} imply that the differentials on the $E_2$-page for the fibre bundle $$\pi_{23}: w_{1^n23}(\P^1)\to \pconf_2(\P^1)$$  cannot, in fact, be isomorphisms. 
\end{remark}

\bibliographystyle{alpha}
\bibliography{VakilWoodConj}
\end{document}